\documentclass[a4paper, 11pt, twoside]{article}
\usepackage{amssymb, amsmath}
\usepackage{amsthm, amsfonts,mathrsfs}
\usepackage[T1]{fontenc}
\usepackage{times}
\usepackage{graphicx}
\usepackage{subfigure}
\usepackage{cite}
\usepackage{epsfig}
\usepackage{color}
\usepackage[all]{xypic}
\newtheorem{theorem}{Theorem}[section]

\newtheorem{definition}{Definition}[section]
\newtheorem{lemma}{Lemma}[section]
\newtheorem{proposition}{Proposition}[section]

\numberwithin{equation}{section}

\graphicspath{{image/}}

\begin{document} 
\title{{\large  \textbf{ Stochastic homogenization for two dimensional Navier--Stokes equations with random  coefficients }}\thanks{\footnotesize{
{Corresponding author. }
			}}}
\author{\small{Dong Su$^a$, Hui Liu$^{b,c}$, Yangyang Shi$^{d*}$ \thanks{dsu224466@163.com, liuhuinanli@126.com, shiyy@chnu.edu.cn.
			}}\\
	\footnotesize{a. School of Mathematics, Nanjing Audit University, Nanjing, Jiangsu 211815, P.R. China;}
\\
	\footnotesize{b. School of Mathematical Sciences, University of Jinan, Jinan,  Shandong 250022, P.R. China;}
\\
	\footnotesize{c. School of Mathematical Sciences, Qufu Normal University, Qufu, Shandong 273165, P.R. China;}
\\
	\footnotesize{d. School of Mathematical and Statistics, Huaibei Normal University, Huaibei  235000, P.R. China;}
}
\date{}
\maketitle
\noindent{\small{\hspace{1.1cm} }}\\
	\\
\noindent \textbf{Abstract~~~}    This paper  derives the  stochastic homogenization for two dimensional Navier--Stokes equations with random  coefficients.  By means of weak convergence method and Stratonovich--Khasminskii averaging principle approach, the solution of  two dimensional Navier--Stokes equations with random  coefficients  converges in distribution to the solution of two  dimensional Navier--Stokes equations with constant coefficients.
\\[2mm]
\textbf{Key words~~~} Stochastic homogenization,  stochastic Navier--Stokes equations, random coefficients.
\\[2mm]
\\
\textbf{2020 Mathematics Subject Classification~~~} 60H15, 60H30, 76D05
\section {Introduction}

The  incompressible Navier--Stokes equations play an important role in  fluid mechanics, atmospheric and ocean dynamics. Over the past many years, the study of the incompressible Navier--Stokes equations has been received much attention. It is given by as follows
\begin{eqnarray*}\label{intr-equ1}
\begin{cases}
{u}_{t}(t)+[-\nu\Delta u(t)+\langle u(t),\nabla\rangle u(t)+\nabla p]=0,\\
\mathrm{div}\; u(t)=0,\\
u(0)=u_0(x),\label{intr-equ1-ini}
\end{cases}
\end{eqnarray*}
where $\nu > 0$ is the viscosity of fluid, the space variable $x=(x_1,x_2)\in \mathbb{T}^2=\mathbb{R}^2/2\pi\mathbb{Z}^2$, the velocity $u(x,t)=(u_1(x,t),u_2(x,t))$ is vector function,  the pressure $p(x,t)$ is   scalar function and $\langle u, \nabla\rangle=u_1 \partial_1+u_2 \partial_2$. To simplify notation, we assume that $\nu=1$ in the following sections.

The fluid dynamics is affected by random forces in the turbulence and statistical hydrodynamics~\cite{Ca,Cha,No,Vi}.  The model of the stochastic Navier--Stokes equations can simulate structural vibrations in aeronautical applications, while the  unknown external forces such as solar heating and industrial pollution can be regarded as random forces in atmospheric dynamics. Inspired by the above facts,  we consider the stochastic homogenization of the  following two--dimensional Navier--Stokes equations for viscous incompressible fluids with random  coefficients defined  on $\mathbb{T}^2$
\begin{eqnarray}\label{main-equ1}
\begin{cases}
{u}_{t}^{\varepsilon}(t)+[-\Delta u^{\varepsilon}(t)+\langle u^\varepsilon(t),\nabla\rangle u^\varepsilon(t)+\nabla p^\varepsilon]=\tilde{q}^{\varepsilon}(x,\omega)u^{\varepsilon}(t)+\tilde{\sigma}^\varepsilon(t,u^\varepsilon(t))dW(t),\\
\mathrm{div}\; u^\varepsilon(t)=0,\\
u^{\varepsilon}(0)=u_0(x).\label{main-equ1-ini}
\end{cases}
\end{eqnarray}
Here $\{W(t)\}$ is a Winner process on complete probability space $(\Omega, \mathcal{F}, \mathbb{P})$, and we denote by $Q$ is the covariance operators of $\{W(t)\}$. The potential $\tilde{q}^{\varepsilon}(x,\omega)=\tilde{q}(\frac{x}{\varepsilon},\omega)$ is highly oscillatory, where $0<\varepsilon<1$ denotes the scale of oscillations. The term $\tilde{\sigma}^\varepsilon(t,u^\varepsilon(t))=\tilde{\sigma}\left(\frac{t}{\varepsilon},u^\varepsilon(t)\right)$ in the system (\ref{main-equ1}).

Lately, the study of homogenization of the Navier--Stokes equations has been intensively studied by many authors. There are some works in this field. The homogenization of the incompressible Navier--Stokes equations in periodic perforated domains was presented  by analytical method~\cite{Al,Al1,Fe,Fe1,Hm1,Lu2,Sp,Ya}.   The homogenization of the compressible Navier--Stokes equations in periodic perforated domains was proved  by similar method~\cite{Be,Hm,Lu,Lu1,Ma,Ne,Ne1,Os1,Po}.  However, stochastic homogenization is a very active study field. One of the study problems in stochastic homogenization is  the partial differential equations with random oscillation coefficients.  The stochastic  homogenization results were presented by analytical method~\cite{Bes1,Bes2,Ha,Su1}, probabilistic approach~\cite{Di,Pa,Su}, and chaos expansion method~\cite{Ba,Ko,Zh}. There are some works on the  homogenization of stochastic partial differential equations.  The authors~\cite{Ic1,Ic2,Ra1,Ra2} concerned mainly stochastic parabolic equations taking place in periodic setting, almost periodic setting and ergodic setting. The homogenization of stochastic partial differential equations by two--scale convergence method~\cite{Wa2,Mo1} and Tartar's oscillating test function approach~\cite{Wa1,Mo2}. There are some nice works on the homogenization of partial differential equations in randomly perforated domains. The authors~\cite{Am1} showed the stochastic homogenization immiscible compressible two--phase flow  by stochastic two--scale convergence method.  Homogenization for the Poisson equations and Stokes equations in randomly perforated domains was studied by the authors~\cite{Gi,Gi1}. Giunti~\cite{Gi2} derived the Darcy's law in randomly perforated domains by analytical method. Homogenization of the linearized ionic transport equations in randomly perforated domains was studied by the authors \cite{Mi}.  It is worth to focus on the stochastic homogenization of the fluid equations.  Wright~\cite{Wr,Wr1} showed the homogenization of incompressible fluids randomly perforated domains by stochastic two--scale convergence in mean  approach. The authors~\cite{Be1,Os} showed the homogenization of the compressible Navier--Stokes equations in randomly perforated domains by constructing Bogovski\u l operator approach. Bessaih \& Maris~\cite{Bes3} proved the homogenization  for the stochastic  Navier--Stokes with a stochastic slip boundary condition by two--scale convergence approach. Shi \& Gao~\cite{Shi} gave the homogenization of stochastic Ginzburg--Landau equation on the half--line with fast boundary fluctuation by analytical method.  The stochastic homogenization of fluid equations with highly oscillating boundary or domains was proved by the above authors. Up to our knowledge, there are very few works for the stochastic homogenization of fluid equation with  random  coefficients.

The novelties are as follow:  due to the fact that the Navier--Stokes equations have strong nonlinearity, it follows that the tightness of solution is difficult to derive directly.  To overcome it, we adopt some  analytical methods to obtain the tightness of solution. It is a fact that the limit of the product of two weakly convergent sequences does not exist. We use the Skorohod theorem and Birkhoff ergodic theorem to solve the problem. Since the rapid oscillation nonlinear coefficient $\tilde{\sigma}^\varepsilon(t,u^\varepsilon(t))$, it is difficult to get directly the limit as $\varepsilon\rightarrow0$.  Utilizing the Stratonovich--Khasminskii averaging principle approach yields the  the limit  of nonlinear coefficient as $\varepsilon\rightarrow0$.

In the following parts, the positive constants $C$ and $C_T$ maybe change from line to line.

The paper is organized as follows: Section \ref{sec:Pr} presents some necessary definitions notations, assumption, and main result (Theorem \ref{main-theom1}).  Section \ref{sec:tig-woutsf} introduces the tightness of solution for the system  (\ref{p-main-equ1}). The  proof of Theorem \ref{main-theom1}  is derived in the last section.
\section {Preliminaries and assumptions}\label{sec:Pr}
We give a complete probability space $(\Omega, \mathcal{F},\mathbb{P})$ and denote by $\mathbb{E}$ the expectation operator with respect to $\mathbb{P}$. The probability space equipped with an ergodic dynamical system $\mathbf{T}_x, x\in\mathbb{R}^n$, that is , a group of measurable maps $\mathbf{T}_x: \Omega\rightarrow\Omega$ such that

$(i)\;\;\mathbf{T}_{x+y}=\mathbf{T}_x\cdot\mathbf{T}_y,\;\; x,y\in \mathbb{R}^d,\;\;\mathbf{T}_{0}={\rm Id}$;

$(ii)\;\;\mathbb{P}(\mathbf{T}_xA)=\mathbb{P}(A)\;\;$ for all$\;\; x\in \mathbb{R}^d,\;\;A\in \mathcal{F}$;

$(iii)\;\mathbf{T}_x(\omega): \mathbb{R}^d\times \Omega\mapsto\Omega$ is a measurable map from $(\mathbb{R}^d\times \Omega,\mathcal{B}\times\mathcal{F}) $ to $(\Omega,\mathcal{F})$ where $\mathcal{B}$ is the Borel $\sigma$-algebra;

$(iv)$   If $A\in\mathcal{F}$  is invariant with respect to $\mathbf{T}_x,\;x\in \mathbb{R}^d$,  then $\mathbb{P}(A)=0$ or $1$\,.\\

\begin{definition}\label{def}
(statistically homogeneous, \cite{Ch,Zhi}) A random field $\tilde{f}(x,\omega),\;x\in\mathbb{R}^n,\;\omega\in\Omega$ is said to be statistically homogeneous if $\tilde{f}(x,\omega)=\tilde{f}(\mathbf{T}_x\omega)$,\;$\tilde{f}(\omega)\in L^2(\Omega)$ where $\mathbf{T}_x$ is a dynamical system in $\Omega$.
\end{definition}

Let $u\in L^2(\mathbb{T}^2;\mathbb{R}^2)$ such that $u$ can be written as a Fourier series, that is
\begin{eqnarray*}\label{pre1-equ}
u(x)=\sum_{s\in \mathbb{Z}^2}u_se^{is\cdot x}
\end{eqnarray*}
with $u_s \in \mathbb{C}^2$ and $u_{-s}=\bar{u}_s$. If $u(x)\in C^\infty(\mathbb{T}^2;\mathbb{R}^2)$, then
\begin{eqnarray*}\label{pre11-equ}
\mathrm{div}\; u(x)=\sum_{s\in \mathbb{Z}^2} is\cdot u_se^{is\cdot x}.
\end{eqnarray*}
We denote by the following space $H$,
\begin{eqnarray*}\label{pre-equ}
H=\left\{u\in L^2(\mathbb{T}^2;\mathbb{R}^2)\mid \mathrm{div}\; u=0\;\; \mathrm{in}\;\;\mathbb{T}^2\;\; \mathrm{and}\;\; \int_{\mathbb{T}^2}u(x)dx=0\right\},
\end{eqnarray*}
with the norm $\|\cdot\|_{L^2}$, the following equation
\begin{eqnarray*}\label{pre-equ1}
H=\left\{u(x)=\sum_{s\in\mathbb{Z}_0^2}u_se^{is\cdot x}\in L^2(\mathbb{T}^2;\mathbb{R}^2)\mid u_{-s}=\bar{u}_s\;\;\mathrm{and}\;\;s\cdot u_s=0\right\}
\end{eqnarray*}
holds, where $\mathbb{Z}_0^2=\mathbb{Z}^2\backslash\{0\}$. The space $H$ is a Hilbert space with inner product $\langle\cdot,\cdot\rangle$.

Next, we introduce a complete orthonormal basis of space $H$ and define
\begin{eqnarray*}\label{pre-equ2}
\mathbb{Z}_{+}^2=\{(s_1,s_2)\mid s_1>0\;\;\mathrm{or}\;\; s_1=0,\;\;s_2>0\},
\end{eqnarray*}
then,
\begin{eqnarray*}\label{pre-equ3}
\mathbb{Z}_{0}^2=\mathbb{Z}_{+}^2\cup(-\mathbb{Z}_{+}^2),\;\;\mathbb{Z}_{+}^2\cap(-\mathbb{Z}_{+}^2)=\emptyset.
\end{eqnarray*}
We define the following set of vectors $\{e_s(x)\mid s\in\mathbb{Z}_{0}^2 \}$,
\begin{eqnarray*}\label{pre-equ4}
e_s(x)=
\begin{cases}
c_s s^{\bot}\sin(s\cdot x),& s\in\mathbb{Z}_{+}^2,\\
c_s s^{\bot}\cos(s\cdot x),& s\in-\mathbb{Z}_{+}^2,
\end{cases}
\end{eqnarray*}
where $c_s=\frac{1}{\sqrt{2}\pi\mid s\mid}$ and if $s=(s_1,s_2)$, then $s^{\bot}=(-s_2,s_1)$. The set $\{e_s(x)\}$ is a Hilbert basis of space $H$.

Let $\Pi$ be the Leray-Helmholtz projection of $L^2(\mathbb{T}^2;\mathbb{R}^2)\rightarrow H$.
\begin{eqnarray*}\label{pre-equ5}
Au:=-\Pi\Delta u
\end{eqnarray*}
is a Stokes operator and $Au=-\Delta u$ in the space-periodic case. $\{e_s(x)\}_{s\in \mathbb{Z}_0^2}$ is a set of eigenvectors of $A$, that is \begin{eqnarray*}\label{pre-equ6}
A e_s=\mid s\mid^2 e_s,\;\;s\in \mathbb{Z}_0^2.
\end{eqnarray*}
For $s\geq0$, we define the spaces as
\begin{eqnarray*}\label{pre-equ7}
H_0^s(\mathbb{T}^2;\mathbb{R}^2)=\left\{u\in H^s(\mathbb{T}^2;\mathbb{R}^2)\mid \int_{\mathbb{T}^2}u(x)dx=0\right\},
\end{eqnarray*}
\begin{eqnarray*}\label{pre-equ8}
V^s=H^s(\mathbb{T}^2;\mathbb{R}^2)\cap H,
\end{eqnarray*}
\begin{eqnarray*}\label{pre-equ9}
X^s=C(0,T;V^s).
\end{eqnarray*}
If $u\in H_0^s(\mathbb{T}^2;\mathbb{R}^2)$, then the norm is defined equivalently as
\begin{eqnarray*}\label{pre-equ10}
\|u\|_s^2=\sum_{k\in\mathbb{Z}_0^2}\mid u_k\mid^2\mid k \mid^{2s},
\end{eqnarray*}
we obtain easily the $\|u\|_s\cong\|A^{\frac{s}{2}}u\|$.

As for the stochastic force, we assume that $\{W(t)\}$ is cylindrical $Q$--Winner processes.  It is written as
\begin{eqnarray*}\label{en1}
W(t)=\sum_{s=1}^\infty Q e_s \beta_s(t),
\end{eqnarray*}
where $\{\beta_s(t)\}_{s\geq 1}$ is  sequence of standard Brownian motion on $(\Omega, \mathcal{F},\mathbb{P})$, $Q: H\rightarrow H$ is bounded linear operator.
For any two separable Hilbert spaces $E$ and $F$, we denote by $\mathcal{L}_2(E,F)$ the space of Hilbert--Schmidt operators from $E$ to $F$. The Hilbert space $\mathcal{L}_2(E,F)$ is  endowed with the inner product
\begin{eqnarray*}\label{en4}
\langle \tilde{A},\tilde{B}\rangle_{\mathcal{L}_2(E,F)}=\mathrm{Tr}_E[\tilde{A}^\top \tilde{B}]=\mathrm{Tr}_F[\tilde{B} \tilde{A}^\top].
\end{eqnarray*}
Now, applying the projection $\Pi$ to system (\ref{main-equ1}) yields
\begin{eqnarray}\label{p-main-equ1}
\begin{cases}
{u}_{t}^{\varepsilon}(t)+[A u^{\varepsilon}(t)+B(u^{\varepsilon}(t))]={q}^{\varepsilon}(x,\omega)u^{\varepsilon}(t)+{\sigma}^\varepsilon(t,u^\varepsilon(t))dW(t),\\
\mathrm{div}\; u^\varepsilon(t)=0,\\
u^{\varepsilon}(0)=u_0(x).\label{p-main-equ1-ini}
\end{cases}
\end{eqnarray}
Here $B(u)=\Pi(\langle u,\nabla \rangle)u$, $qu=\Pi(\tilde{q}u)$ and $\sigma=\Pi(\tilde{\sigma})$. Next, we give the following assumptions of equations (\ref{p-main-equ1}).

$(\mathbf{H_1})$ ${q}(x,\omega), x\in \mathbb{R}^d$   is stationary ergodic and statistically homogeneous random field on
$(\Omega, \mathcal{F},\mathbb{P})$,
that is
\begin{equation*}
{q}(x,\omega)={q}{(\mathbf{T}_x\omega)}.
\end{equation*}

$(\mathbf{H_2}) $  For  all $\omega \in \Omega$, there is a constant $C> 0$\,such that
\begin{eqnarray*}\label{q}
|q(x,\omega)|\leq C.
\end{eqnarray*}

$(\mathbf{H_3}) $ The ${\sigma}(t,h),\partial_h({\sigma}(t,h)),\partial_h^2({\sigma}(t,h))\in \mathcal{L}_2(H_{Q},H),t\in(0,T),H_{Q}=QH$ are  bounded, that is for all $T>0$,
\begin{eqnarray*}\label{en5}
\sup_{t\in(0,T), h\in H}\|{\sigma}(t,h)\|_{\mathcal{L}_2(H_{Q},H)}<\infty,
\end{eqnarray*}
\begin{eqnarray*}\label{en51}
\sup_{t\in(0,T), h\in H}\|\partial_h({\sigma}(t,h))\|_{\mathcal{L}_2(H_{Q},H)}<\infty,
\end{eqnarray*}
and
\begin{eqnarray*}\label{en52}
\sup_{t\in(0,T), h\in H}\|\partial_h^2({\sigma}(t,h))\|_{\mathcal{L}_2(H_{Q},H)}<\infty.
\end{eqnarray*}
 For all $h_1,h_2 \in H$,
\begin{eqnarray*}\label{en6}
\|{\sigma}(t,h_1)-{\sigma}(t,h_2)\|_{\mathcal{L}_2(H_{Q},H)}^2\leq L\|h_1-h_2\|_{H}^2,
\end{eqnarray*}
where $L > 0$ is a Lipschitz constant. The ${\sigma}(t,\cdot)$ is periodic with respect to the variable $t$. The well--posedness of systems (\ref{p-main-equ1}) can be proved by the Galerkin approach~\cite{Ku1} and the monotone method~\cite{Me}.

Then we present our main results.
\begin{theorem}\label{main-theom1}
Under the assumptions $(\mathbf{H_{1}})$--$(\mathbf{H_{3}})$. For every  initial conditions $u_0\in V^2$,  for any $T>0$, the solution $\{u^{\varepsilon}(\cdot,\cdot)\}$ to equations (\ref{p-main-equ1}) converges in distribution to $u(\cdot,\cdot)$ in space $C(0,T;X^1)$, as $\varepsilon \to 0$, $u(\cdot,\cdot)$ is a solution of the following equation
\begin{eqnarray*}\label{eff-equ1}
\begin{cases}
{u}_{t}(t)+[A u(t)+B(u(t))]=\bar{q}u(t)+\bar{\sigma}(u(t))dW(t),\\
\mathrm{div}\; u(t)=0,\\
u(0)=u_0(x),\label{eff-equ1-ini}
\end{cases}
\end{eqnarray*}
where $\{W(t)\}$ is cylindrical Wiener processes, $\bar{\sigma}(\cdot)=\frac{1}{T}\int_0^T\sigma_i(t,\cdot)dt$ and $\bar{q}=\mathbb{E}q(0,\omega).$
\end{theorem}

We need the following results~\cite{Ku1}.
\begin{lemma}\label{lemma-ine}
(Ladyzhenskaya inequality) Let $u\in H^1(D)$ such that
\begin{eqnarray*}\label{lemma-ine-equ}
\|u\|_{L^4}\leq C\|u\|_{\frac{1}{2}}\leq C \sqrt{\|u\|\
\|u\|_1},
\end{eqnarray*}
where $C$ is a constant.
\end{lemma}
\begin{lemma}\label{lemma-ine1}
(Interpolation inequality) Let $a < b$ be any real number and $0\leq \theta\leq 1$ be a constant such that
\begin{eqnarray*}\label{lemma-ine1-equ}
\|u\|_{\theta a+(1-\theta)b}\leq \|u\|_{a}^\theta\|u\|_b^{1-\theta},
\end{eqnarray*}
where $u\in H^b(D)$.
\end{lemma}
\begin{lemma}\label{lemma-b}
We set $B(u,v)=\Pi(\langle u,\nabla\rangle)v$, so, $B(u)=B(u,u)$. If $u,v,w\in C^\infty(\mathbb{T}^2)\cap H$, then

$(i)~(B(u,v),v)=0,\;\;(B(u,v),w)= -(B(u,w),v).$

$(ii)~(B(u),\Delta u)=0.$

$(iii)$~For all $m\in\mathbb{N}$, we have
\begin{eqnarray*}\label{lemma-b-equ}
\mid (A^m u, B(u))\mid\leq C\|u\|_{m+1}^{\frac{4m-1}{2m}}\|u\|_1^{\frac{m+1}{2m}}\|u\|^{\frac{1}{2}}.
\end{eqnarray*}

Furthermore, for all $\delta> 0$, there exists a positive constant $C$ relating to $m$ and $\delta$ such that
\begin{eqnarray*}\label{lemma-b-equ1}
\mid (A^m u, B(u))\mid\leq \delta\|u\|_{m+1}^{2}+C(m,\delta)\|u\|_1^{2m+2}\|u\|^{2m}.
\end{eqnarray*}

$(iv)$~ \begin{eqnarray*}\label{pre-equ4}
\mid (B(u,v),w)\mid\leq
\begin{cases}
C\|u\|_{\frac{1}{2}}\|v\|_{\frac{1}{2}}\|w\|_1,\\
C\|u\|_{L^4}\|v\|_1\|w\|^{\frac{1}{2}}\|w\|_1^{\frac{1}{2}}.
\end{cases}
\end{eqnarray*}

$(v)$~There exists a positive constant $C$  such that
\begin{eqnarray*}\label{lemma-b-equ2}
\|B(u)-B(v)\|_{-1}\leq C(\|u\|_1+\|v\|_1)\|u-v\|_1.
\end{eqnarray*}
\end{lemma}

We  need the following compact embedding result~\cite{Si}.
\begin{lemma}\label{lem4}
Assume  that $E, E_0,$ and $E_1$ are Banach space such that $E_1$ is compacted embedded into~$E_0$, the interpolation space $(E_0,E_1)_{\theta,1}\subset E$ with $\theta \in (0,1)$ and $E \subset E_0$. Suppose $p_0,p_1\in [0,\infty]$ and $T > 0$, such that
\begin{center}
$\mathcal{V}$\;\;\;\;\; is bounded set in \;\;\;\;\;$L^p(0,T; E_1)$
\end{center}
and
\begin{center}
$\partial\mathcal{V}:=\{\partial v : v\in\mathcal{V}\}$\;\;\;\;is bounded set in \;\;\;\;\;$L^{p_0}(0,T; E_0)$.
\end{center}
Here $\partial$ denotes the distribution derivative. If $1-\theta > \frac{1}{p_\theta}$ with
\begin{eqnarray*}
\frac{1}{p_\theta}=\frac{1-\theta}{p_0}+\frac{\theta}{p_1}.
\end{eqnarray*}
Then $\mathcal{V}$ is relatively compact in $C(0,T; E)$.
\end{lemma}
To prove Lemma \ref{pf-lem}, we also need the following result~\cite{Li}.

\begin{lemma}\label{lem12}
For all $T>0$. Let $\mathcal{Q}$ be bounded region in $D\times [0,T]$. For all given functions $g^\varepsilon$ and $g$ in $L^p(\mathcal{Q}) (1< p<\infty)$, if
\begin{eqnarray*}\label{lem122}
|g^\varepsilon|_{L^{p}(\mathcal{Q})}\leq C \quad \text{and } \quad g^\varepsilon\rightarrow g\; \rm{in}\,\,\mathcal{Q}\,\,\rm{ almost\,\, everywhere},
\end{eqnarray*}
for some positive constant $C$, then $g^\varepsilon$ converges weakly to $g$ in $L^p(\mathcal{Q})$.
\end{lemma}


\section{Tightness for $\{u^\varepsilon(\cdot,\cdot)\}_\varepsilon$}\label{sec:tig-woutsf}
In this section, we  consider the tightness of the solution  of two dimensional stochastic Navier--Stokes equations (\ref{p-main-equ1}).
\begin{proposition}\label{tight-theo}
 Assume that $(\mathbf{H_{1}})$--$(\mathbf{H_{3}})$,  and $\mathbb{E}\|u_0\|_{V^1}^2\leq C$ for some constant $C >0$ hold. Then
\begin{eqnarray}\label{theo-equ}
\mathbb{E}\sup_{0\leq t\leq T}\|u^\varepsilon(t)\|^2+2\mathbb{E}\int_0^T\|u^\varepsilon(s)\|_1^2ds\leq C_T(1+\mathbb{E}\|u_0\|^2),
\end{eqnarray}
and
\begin{eqnarray}\label{theo-eequ}
\mathbb{E}\sup_{0\leq t\leq T}\|u^\varepsilon(t)\|_1^2\leq C_T(1+\mathbb{E}\|u_0\|_1^2),
\end{eqnarray}
\end{proposition}
\begin{proof}
By It\^{o}'s formula,
\begin{eqnarray*}\label{theo-equ1}
&&\frac{1}{2}\frac{d}{dt}\|u^\varepsilon(t)\|^2=\langle u^\varepsilon(t),-(Au^\varepsilon(t)+B(u^\varepsilon(t)))\rangle\\
&&\quad\quad\quad\quad\quad\quad\quad+\langle u^\varepsilon(t),q^\varepsilon(x,\omega)u^\varepsilon(t)\rangle+\langle u^\varepsilon(t),\sigma^\varepsilon(t,u^\varepsilon(t))dW(t)\rangle\\&&\quad\quad\quad\quad\quad\quad\quad+\frac{1}{2}\|\sigma^\varepsilon(t,u^\varepsilon(t))\|_{\mathcal{L}_2^Q}^2,
\end{eqnarray*}
and
\begin{eqnarray*}\label{theo-eequ1}
&&\frac{1}{2}\frac{d}{dt}\|u^\varepsilon(t)\|_1^2=\langle\Delta u^\varepsilon(t),(Au^\varepsilon(t)+B(u^\varepsilon(t)))\rangle\\
&&\quad\quad\quad\quad\quad\quad\quad-\langle \Delta u^\varepsilon(t),q^\varepsilon(x,\omega)u^\varepsilon(t)\rangle-\langle \Delta u^\varepsilon(t),\sigma^\varepsilon(t,u^\varepsilon(t))dW(t)\rangle\\&&\quad\quad\quad\quad\quad\quad\quad+\frac{1}{2}\|\partial_x\sigma^\varepsilon(t,u^\varepsilon(t))\nabla u^\varepsilon(t)\|_{\mathcal{L}_2^Q}^2.
\end{eqnarray*}
Furthermore, the Young inequality and $(\mathbf{H_{2}})$ yield
\begin{eqnarray}\label{theo-equ2}
&&\frac{1}{2}\frac{d}{dt}\|u^\varepsilon(t)\|^2+\|u^\varepsilon(t)\|_1^2\leq C \|u^\varepsilon(t)\|^2+\langle u^\varepsilon(t),\sigma^\varepsilon(t,u^\varepsilon(t))dW(t)\rangle\nonumber\\&&\quad\quad\quad\quad\quad\quad\quad\quad\quad\quad\quad\quad+\frac{1}{2}\|\sigma^\varepsilon(t,u^\varepsilon(t))\|_{\mathcal{L}_2^Q}^2,
\end{eqnarray}
and
\begin{eqnarray}\label{theo-eequ2}
&&\frac{1}{2}\frac{d}{dt}\|u^\varepsilon(t)\|_1^2\leq -\frac{1}{2}\|\Delta u^\varepsilon(t)\|^2+\frac{C}{2} \|u^\varepsilon(t)\|^2\nonumber\\&&\quad\quad\quad\quad\quad\quad\quad\quad\quad\quad\quad\quad+\langle \Delta u^\varepsilon(t),\sigma^\varepsilon(t,u^\varepsilon(t))dW(t)\rangle\nonumber\\&&\quad\quad\quad\quad\quad\quad\quad\quad\quad\quad\quad\quad+\frac{1}{2}\|\partial_x\sigma^\varepsilon(t,u^\varepsilon(t))\nabla u^\varepsilon(t)\|_{\mathcal{L}_2^Q}^2.
\end{eqnarray}

Integrating from $0$ to $t$ of equations (\ref{theo-equ2}) and (\ref{theo-eequ2}), respectively, yields
\begin{eqnarray}\label{theo-equ3}
&&\|u^\varepsilon(t)\|^2+2\int_0^t \|u^\varepsilon(s)\|_1^2ds\leq\|u_0\|^2+C\int_0^t\|u^\varepsilon(s)\|^2ds\nonumber \\&&\quad\quad\quad\quad\quad\quad\quad\quad\quad\quad\quad\quad\quad+2\int_0^t\langle u^\varepsilon(s),\sigma^\varepsilon(s,u^\varepsilon(s))\rangle dW(s)\nonumber\\&&\quad\quad\quad\quad\quad\quad\quad\quad\quad\quad\quad\quad\quad+\int_0^t\|\sigma^\varepsilon(s,u^\varepsilon(s))\|_{\mathcal{L}_2^Q}^2ds,
\end{eqnarray}
and
\begin{eqnarray}\label{theo-eequ3}
&&\|u^\varepsilon(t)\|_1^2+\int_0^t \|\Delta u^\varepsilon(s)\|^2ds\leq\|u_0\|_1^2+C\int_0^t\|u^\varepsilon(s)\|_1^2ds\nonumber \\&&\quad\quad\quad\quad\quad\quad\quad\quad\quad\quad\quad\quad\quad+2\int_0^t\langle \Delta u^\varepsilon(s),\sigma^\varepsilon(s,u^\varepsilon(s))\rangle dW(s)\nonumber\\&&\quad\quad\quad\quad\quad\quad\quad\quad\quad\quad\quad\quad\quad+\int_0^t\|\partial_x\sigma^\varepsilon(s,u^\varepsilon(s))\nabla u^\varepsilon(t)\|_{\mathcal{L}_2^Q}^2ds.
\end{eqnarray}
Furthermore,
\begin{eqnarray}\label{theo-equ4}
&&\sup_{0\leq t\leq T}\|u^\varepsilon(t)\|^2+2\int_0^T \|u^\varepsilon(s)\|_1^2ds\nonumber\\&&\quad\quad\quad\quad\leq\|u_0\|^2+C\int_0^T\sup_{0\leq\tau\leq s}\|u^\varepsilon(\tau)\|^2ds\nonumber\\&&\quad\quad\quad\quad\quad+2\sup_{0\leq t \leq T}\int_0^t\langle u^\varepsilon(s),\sigma^\varepsilon(s,u^\varepsilon(s))\rangle dW(s)\nonumber\\&&\quad\quad\quad\quad\quad+\int_0^T\|\sigma^\varepsilon(s,u^\varepsilon(s))\|_{\mathcal{L}_2^Q}^2ds,
\end{eqnarray}
and
\begin{eqnarray}\label{theo-eequ4}
&&\sup_{0\leq t\leq T}\|u^\varepsilon(t)\|_1^2+\int_0^T \| \Delta u^\varepsilon(s)\|^2ds\nonumber\\&&\quad\quad\quad\quad\leq\|u_0\|_1^2+C\int_0^T\sup_{0\leq\tau\leq s}\|u^\varepsilon(\tau)\|_1^2ds\nonumber\\&&\quad\quad\quad\quad\quad+2\sup_{0\leq t \leq T}\int_0^t\langle \Delta u^\varepsilon(s),\sigma^\varepsilon(s,u^\varepsilon(s))\rangle dW(s)\nonumber\\&&\quad\quad\quad\quad\quad+\int_0^T\|\partial_x\sigma^\varepsilon(s,u^\varepsilon(s))\nabla u^\varepsilon(s) \|_{\mathcal{L}_2^Q}^2ds.
\end{eqnarray}
By  Burkholder--Davis--Gundy inequality, Young inequality, H\"{o}lder inequality and $(\mathbf{H_{3}})$,
\begin{eqnarray}\label{theo-equ41}
&&\mathbb{E}\sup_{0\leq t \leq T}\int_0^t\langle u^\varepsilon(s),\sigma^\varepsilon(s,u^\varepsilon(s))\rangle dW(s)\nonumber\\&&\quad\quad\quad\quad\quad\quad\leq
\mathbb{E}\left(\int_0^T\langle u^\varepsilon(s),\sigma^\varepsilon(s,u^\varepsilon(s))\rangle^2ds\right)^{\frac{1}{2}}\nonumber\\&&\quad\quad\quad\quad\quad\quad
\leq \mathbb{E}\left(\int_0^T \|u^\varepsilon(s)\|^2 \|\sigma^\varepsilon(s,u^\varepsilon(s))\|_{\mathcal{L}_2^Q}^2ds\right)^{\frac{1}{2}}\nonumber\\&&\quad\quad\quad\quad\quad\quad\leq \mathbb{E} \left(\sup_{0\leq s\leq T}\|\sigma^\varepsilon(s,u^\varepsilon(s))\|_{\mathcal{L}_2^Q}^2\int_0^T \|u^\varepsilon(s)\|^2 ds\right)^{\frac{1}{2}}\nonumber\\&&\quad\quad\quad\quad\quad\quad\leq
\frac{1}{2}\left(C+\mathbb{E}\int_0^T \sup_{0\leq\tau \leq s}\|u^\varepsilon(\tau)\|^2 ds\right),
\end{eqnarray}
and
\begin{eqnarray}\label{theo-eequ41}
&&\mathbb{E}\sup_{0\leq t \leq T}\int_0^t\langle \Delta u^\varepsilon(s),\sigma^\varepsilon(s,u^\varepsilon(s))\rangle dW(s)\nonumber\\&&\quad\quad\quad\quad\quad\quad\leq
\mathbb{E}\left(\int_0^T\langle \Delta u^\varepsilon(s),\sigma^\varepsilon(s,u^\varepsilon(s))\rangle^2ds\right)^{\frac{1}{2}}\nonumber\\&&\quad\quad\quad\quad\quad\quad
\leq \mathbb{E}\left(\int_0^T \|\Delta u^\varepsilon(s)\|^2 \|\sigma^\varepsilon(s,u^\varepsilon(s))\|_{\mathcal{L}_2^Q}^2ds\right)^{\frac{1}{2}}\nonumber\\&&\quad\quad\quad\quad\quad\quad\leq \mathbb{E} \left(\sup_{0\leq s\leq T}\|\sigma^\varepsilon(s,u^\varepsilon(s))\|_{\mathcal{L}_2^Q}^2\int_0^T \|\Delta u^\varepsilon(s)\|^2 ds\right)^{\frac{1}{2}}\nonumber\\&&\quad\quad\quad\quad\quad\quad\leq
\frac{1}{2}\left(C+\mathbb{E}\int_0^T \|\Delta u^\varepsilon(s)\|^2 ds\right).
\end{eqnarray}
Substituting (\ref{theo-equ41}), (\ref{theo-eequ41}) into (\ref{theo-equ4}), (\ref{theo-eequ4}), respectively,  yields
\begin{eqnarray*}\label{theo-equ42}
\mathbb{E}\sup_{0\leq t\leq T}\|u^\varepsilon(t)\|^2+2\mathbb{E}\int_0^T \|u^\varepsilon(s)\|_1^2ds\leq C_T(\mathbb{E}\|u_0\|^2+1)+C\mathbb{E}\int_0^T \sup_{0\leq\tau \leq s}\|u^\varepsilon(\tau)\|^2 ds,
\end{eqnarray*}
and
\begin{eqnarray*}\label{theo-eequ42}
\mathbb{E}\sup_{0\leq t\leq T}\|u^\varepsilon(t)\|_1^2\leq C_T(\mathbb{E}\|u_0\|_1^2+1)+C\mathbb{E}\int_0^T \sup_{0\leq\tau \leq s}\|u^\varepsilon(\tau)\|_1^2 ds.
\end{eqnarray*}
By  Gronwall's inequality, we obtain (\ref{theo-equ}) and (\ref{theo-eequ}).
\end{proof}
\begin{proposition}\label{tight-pro}
 Assume that $(\mathbf{H_{1}})$--$(\mathbf{H_{3}})$,  and $\mathbb{E}\|u_0\|_{V^1}^{2p}\leq C$ for some constant $C >0$ and any $p>0$ hold. Then
\begin{eqnarray}\label{theo-pro-equ}
\mathbb{E}\sup_{0\leq t\leq T}\|u^\varepsilon(t)\|^{2p}\leq C_T(\mathbb{E}\|u_0\|^{2p}+C)+1,
\end{eqnarray}
and
\begin{eqnarray}\label{theo-pro-eequ}
\mathbb{E}\sup_{0\leq t\leq T}\|u^\varepsilon(t)\|_1^{2p}\leq C_T(\mathbb{E}\|u_0\|_1^{2p}+C)+1.
\end{eqnarray}

\end{proposition}
\begin{proof}
We first prove that for any $p>1$
\begin{eqnarray}\label{theo-pro-equ1}
\mathbb{E}\sup_{0\leq t\leq T}\|u^\varepsilon(t)\|^{2p}\leq C_T(\mathbb{E}\|u_0\|^{2p}+C),
\end{eqnarray}
and
\begin{eqnarray}\label{theo-pro-eequ1}
\mathbb{E}\sup_{0\leq t\leq T}\|u^\varepsilon(t)\|_1^{2p}\leq C_T(\mathbb{E}\|u_0\|_1^{2p}+C).
\end{eqnarray}
From (\ref{theo-equ3}) and (\ref{theo-eequ3}), we deduce
\begin{eqnarray}\label{theo-pro-equ2}
&&\mathbb{E}\sup_{0\leq t\leq T}\|u^\varepsilon(t)\|^{2p}\nonumber\\&&\quad\quad\leq C_{T}(\mathbb{E}\|u_0\|^{2p}+\mathbb{E}\sup_{0\leq t\leq T}\left(\int_0^t\|u^\varepsilon(s)\|^2ds\right)^p\nonumber\\&&\quad\quad\quad\quad+\mathbb{E}\sup_{0\leq t \leq T} \left| \int_0^t\langle u^\varepsilon(s),\sigma^\varepsilon(s,u^\varepsilon(s))\rangle dW(s)\right|^p+1),
\end{eqnarray}
and
\begin{eqnarray}\label{theo-pro-eequ2}
&&\mathbb{E}\sup_{0\leq t\leq T}\|u^\varepsilon(t)\|_1^{2p}\nonumber\\&&\quad\quad\leq C_{T}(\mathbb{E}\|u_0\|_1^{2p}+\mathbb{E}\sup_{0\leq t\leq T}\left(\int_0^t\|u^\varepsilon(s)\|_1^2ds\right)^p\nonumber\\&&\quad\quad\quad\quad+\mathbb{E}\sup_{0\leq t \leq T} \left| \int_0^t\langle \Delta u^\varepsilon(s),\sigma^\varepsilon(s,u^\varepsilon(s))\rangle dW(s)\right|^p+1).
\end{eqnarray}

The  Burkholder--Davis--Gundy inequality, Young inequality, H\"{o}lder inequality and $(\mathbf{H_{3}})$ yield
\begin{eqnarray*}\label{theo-pro-equ3}
&&\mathbb{E}\sup_{0\leq t \leq T} \left| \int_0^t\langle u^\varepsilon(s),\sigma^\varepsilon(s,u^\varepsilon(s))\rangle dW(s)\right|^p\\&&\quad\quad\quad\quad
\leq \mathbb{E}\left(\int_0^T\langle u^\varepsilon(s),\sigma^\varepsilon(s,u^\varepsilon(s))\rangle^2ds \right)^{\frac{p}{2}}
\\&&\quad\quad\quad\quad\leq \mathbb{E}\left(\int_0^T \|u^\varepsilon(s)\|^2\|\sigma^\varepsilon(s,u^\varepsilon(s))\|_{\mathcal{L}_2^Q}^2 ds\right )^{\frac{p}{2}}\\&&\quad\quad\quad\quad
\leq\mathbb{E}\left(\sup_{0\leq s\leq T}\|\sigma^\varepsilon(s,u^\varepsilon(s))\|_{\mathcal{L}_2^Q}^2\int_0^T \|u^\varepsilon(s)\|^2 ds\right )^{\frac{p}{2}}\\&&\quad\quad\quad\quad\leq C \mathbb{E} \left(\int_0^T \|u^\varepsilon(s)\|^2 ds\right)^{\frac{p}{2}},
\end{eqnarray*}
and
\begin{eqnarray*}\label{theo-pro-eequ3}
&&\mathbb{E}\sup_{0\leq t \leq T} \left| \int_0^t\langle \Delta u^\varepsilon(s),\sigma^\varepsilon(s,u^\varepsilon(s))\rangle dW(s)\right|^p\\
&&\quad\quad\quad\quad
\leq \mathbb{E}\left(\int_0^T\langle\Delta u^\varepsilon(s),\sigma^\varepsilon(s,u^\varepsilon(s))\rangle^2ds \right)^{\frac{p}{2}}
\\&&\quad\quad\quad\quad\leq \mathbb{E}\left(\int_0^T \|\nabla u^\varepsilon(s)\|^2\|\partial_x\sigma^\varepsilon(s,u^\varepsilon(s)\nabla u^\varepsilon(s)\|_{\mathcal{L}_2^Q}^2 ds\right )^{\frac{p}{2}}\\&&\quad\quad\quad\quad
\leq\mathbb{E}\left(\sup_{0\leq s\leq T}\|\partial_x\sigma^\varepsilon(s,u^\varepsilon(s)\nabla u^\varepsilon(s)\|_{\mathcal{L}_2^Q}^2\int_0^T \|u^\varepsilon(s)\|_1^2 ds\right )^{\frac{p}{2}}\\&&\quad\quad\quad\quad\leq C \mathbb{E} \left(\int_0^T \|u^\varepsilon(s)\|_1^2 ds\right)^{\frac{p}{2}}.
\end{eqnarray*}
Utilizing Young inequality and H\"{o}lder inequality  yields
\begin{eqnarray*}\label{theo-pro-equ4}
&&\left(\int_0^T \|u^\varepsilon(s)\|^2 ds\right)^{\frac{p}{2}}\\&&\quad\quad\leq \frac{1}{2}+\frac{1}{2}\left(\int_0^T \|u^\varepsilon(s)\|^2 ds\right)^{p}\\&&\quad\quad\leq\frac{1}{2}+\frac{C_T}{2}\int_0^T \|u^\varepsilon(s)\|^{2p} ds,
\end{eqnarray*}
and
\begin{eqnarray*}\label{theo-pro-eequ4}
&&\left(\int_0^T \|u^\varepsilon(s)\|_1^2 ds\right)^{\frac{p}{2}}\\&&\quad\quad\leq \frac{1}{2}+\frac{1}{2}\left(\int_0^T \|u^\varepsilon(s)\|_1^2 ds\right)^{p}\\&&\quad\quad\leq\frac{1}{2}+\frac{C_T}{2}\int_0^T \|u^\varepsilon(s)\|_1^{2p} ds.
\end{eqnarray*}
Furthermore,
\begin{eqnarray}\label{theo-pro-equ5}
\mathbb{E}\sup_{0\leq t \leq T} \left| \int_0^t\langle u^\varepsilon(s),\sigma^\varepsilon(s,u^\varepsilon(s))\rangle dW(s)\right|^p \leq
\frac{C}{2}+\frac{C_T}{2}\mathbb{E}\int_0^T \|u^\varepsilon(s)\|^{2p}ds,
\end{eqnarray}
and
\begin{eqnarray}\label{theo-pro-eequ5}
\mathbb{E}\sup_{0\leq t \leq T} \left| \int_0^t\langle \Delta u^\varepsilon(s),\sigma^\varepsilon(s,u^\varepsilon(s))\rangle dW(s)\right|^p \leq
\frac{C}{2}+\frac{C_T}{2}\mathbb{E}\int_0^T \|u^\varepsilon(s)\|_1^{2p}ds,
\end{eqnarray}
By H\"{o}lder inequality, we have
\begin{eqnarray}\label{theo-pro-equ6}
\mathbb{E}\sup_{0\leq t\leq T}\left(\int_0^t\|u^\varepsilon(s)\|^2ds\right)^p \leq C_T \mathbb{E}\int_0^T \|u^\varepsilon(s)\|^{2p}ds.
\end{eqnarray}
and
\begin{eqnarray}\label{theo-pro-eequ6}
\mathbb{E}\sup_{0\leq t\leq T}\left(\int_0^t\|u^\varepsilon(s)\|_1^2ds\right)^p \leq C_T \mathbb{E}\int_0^T \|u^\varepsilon(s)\|_1^{2p}ds.
\end{eqnarray}
Combining (\ref{theo-pro-equ5}), (\ref{theo-pro-equ6}) with (\ref{theo-pro-equ2}) yields
\begin{eqnarray*}\label{theo-pro-equ7}
\mathbb{E}\sup_{0\leq t\leq T}\|u^\varepsilon(t)\|^{2p}\leq C_{T}(\mathbb{E}\|u_0(x)\|^{2p}+C)+C_T \mathbb{E}\int_0^T \sup_{0\leq \tau\leq s}\|u^\varepsilon(\tau)\|^{2p}ds.
\end{eqnarray*}
Substituting (\ref{theo-pro-eequ5}), (\ref{theo-pro-eequ6}) into (\ref{theo-pro-eequ2}) yields
\begin{eqnarray*}\label{theo-pro-eequ7}
\mathbb{E}\sup_{0\leq t\leq T}\|u^\varepsilon(t)\|_1^{2p}\leq C_{T}(\mathbb{E}\|u_0(x)\|_1^{2p}+C)+C_T \mathbb{E}\int_0^T \sup_{0\leq \tau\leq s}\|u^\varepsilon(\tau)\|_1^{2p}ds.
\end{eqnarray*}
The Gronwall's inequality yields (\ref{theo-pro-equ1}) and (\ref{theo-pro-eequ1}).

For the case $0< p\leq 1$, we obtain
\begin{eqnarray*}\label{theo-pro-equ8}
\|u^\varepsilon(s)\|^{2p}\leq \frac{1}{p'}\|u^\varepsilon(s)\|^{2pp'}+ \frac{1}{q'},
\end{eqnarray*}
and
\begin{eqnarray*}\label{theo-pro-eequ8}
\|u^\varepsilon(s)\|_1^{2p}\leq \frac{1}{p'}\|u^\varepsilon(s)\|_1^{2pp'}+ \frac{1}{q'}.
\end{eqnarray*}
Here we have used Young inequality
\begin{eqnarray*}\label{theo-pro-equ9}
ab\leq \frac{1}{p'} a^{p'}+\frac{1}{q'} b^{q'},
\end{eqnarray*}
where $p',q'>1$\;,$\frac{1}{p'}+\frac{1}{q'}=1$.
Selecting $p'=2p^{-1}$ yields
\begin{eqnarray*}\label{theo-pro-equ10}
\|u^\varepsilon(s)\|^{2p}\leq \frac{p}{2}\|u^\varepsilon(s)\|^4+1,
\end{eqnarray*}
and
\begin{eqnarray*}\label{theo-pro-eequ10}
\|u^\varepsilon(s)\|_1^{2p}\leq \frac{p}{2}\|u^\varepsilon(s)\|_1^4+1.
\end{eqnarray*}
It follows from (\ref{theo-pro-equ1}) and (\ref{theo-pro-eequ1}) that
\begin{eqnarray*}\label{theo-pro-equ11}
\mathbb{E}\sup_{0\leq s\leq T}\|u^\varepsilon(s)\|^{4}\leq C_T(\mathbb{E}\|u_0\|^{4}+C),
\end{eqnarray*}
and
\begin{eqnarray*}\label{theo-pro-eequ11}
\mathbb{E}\sup_{0\leq s\leq T}\|u^\varepsilon(s)\|_1^{4}\leq C_T(\mathbb{E}\|u_0\|_1^{4}+C).
\end{eqnarray*}
The above estimations yield (\ref{theo-pro-equ}) and (\ref{theo-pro-eequ}).
\end{proof}
\begin{proposition}\label{tight-2theo}
 Assume that $(\mathbf{H_{1}})$--$(\mathbf{H_{2}})$,  and $\mathbb{E}\|u_0\|_{V^2}^2\leq C$ for some constant $C >0$ hold. Then, for $0< \delta << \frac{1}{2}$
\begin{eqnarray}\label{theo-2equ}
\mathbb{E}\sup_{0\leq t\leq T}\|u^\varepsilon(t)\|_2^2+2(1-2\delta)\mathbb{E}\int_0^T\|u^\varepsilon(s)\|_3^2ds\leq C_{T}(\mathbb{E}\|u_0\|_2^2+1).
\end{eqnarray}
\end{proposition}
\begin{proof}
The It\^{o}'s formula yields
\begin{eqnarray}\label{theo-2equ1}
&&\frac{1}{2}\frac{d}{dt}\|u^\varepsilon(t)\|_2^2=-\langle \Delta^2 u^\varepsilon(t), Au^\varepsilon(t)+B(u^\varepsilon(t))\rangle\nonumber\\&&\quad\quad\quad\quad\quad\quad\quad+\langle \Delta^2 u^\varepsilon(t) ,q^\varepsilon(x,\omega)u^\varepsilon(t)\rangle\nonumber\\&&\quad\quad\quad\quad\quad\quad\quad+\langle \Delta^2 u^\varepsilon(t) ,\sigma^\varepsilon(t,u^\varepsilon(t))dW(t)\rangle\nonumber\\&&\quad\quad\quad\quad\quad\quad\quad+\frac{1}{2}\|\partial_x^2(t,u^\varepsilon(t))\Delta u^\varepsilon(t)\|_{\mathcal{L}_2^Q}^2
\end{eqnarray}
By $(iii)$ of Lemma \ref{lemma-b}, for all $\delta>0$, we have
\begin{eqnarray}\label{theo-2equ2}
\mid \langle A^2 u^\varepsilon(t),B(u^\varepsilon(t))\rangle \mid\leq
\delta\|u^\varepsilon(t)\|_3^2+C(\delta)\|u^\varepsilon(t)\|_1^6\|u^\varepsilon(t)\|^4.
\end{eqnarray}
With the help Young inequality and $(\mathbf{H_2}) $, we deduce
\begin{eqnarray}\label{theo-2equ3}
&&\mid \langle \Delta^2 u^\varepsilon(t),q^\varepsilon(x,\omega)u^\varepsilon(t)\rangle \mid\nonumber\\&&\quad\quad\leq
\delta\|u^\varepsilon(t)\|_3^2+C(\delta)\|u^\varepsilon(t)\|_1^2\nonumber\\&&\quad\quad\leq
\delta\|u^\varepsilon(t)\|_3^2+C(\delta)\|u^\varepsilon(t)\|_2^2.
\end{eqnarray}
Combining (\ref{theo-2equ2}), (\ref{theo-2equ3}) with (\ref{theo-2equ1}) yields
\begin{eqnarray}\label{theo-2equ4}
&&\frac{1}{2}\frac{d}{dt}\|u^\varepsilon(t)\|_2^2\leq (2\delta-1)\|u^\varepsilon(t)\|_3^2+C(\delta)\|u^\varepsilon(t)\|_1^6\|u^\varepsilon(t)\|^4+C(\delta)\|u^\varepsilon(t)\|_2^2\nonumber\\&&\quad\quad\quad\quad\quad\quad\quad+\langle \Delta^2 u^\varepsilon(t) ,\sigma^\varepsilon(t,u^\varepsilon(t))dW(t)\rangle+\frac{1}{2}\|\partial_x^2(t,u^\varepsilon(t))\Delta u^\varepsilon(t)\|_{\mathcal{L}_2^Q}^2.
\end{eqnarray}
Taking $0<\delta<<\frac{1}{2}$ and integrating from $0$ to $t$ of equation (\ref{theo-2equ4}) yields
\begin{eqnarray}\label{theo-2equ5}
&&\|u^\varepsilon(t)\|_2^2+2(1-2\delta)\int_0^t\|u^\varepsilon(s)\|_3^2ds\nonumber\\&&\quad\leq \|u_0\|_2^2 +2C(\delta)\int_
0^t\|u^\varepsilon(s)\|_1^6\|u^\varepsilon(s)\|^4ds+2C(\delta)\int_0^t\|u^\varepsilon(s)\|_2^2ds\nonumber\\&&\quad\quad+2\int_0^t\langle \Delta^2 u^\varepsilon(s) ,\sigma^\varepsilon(s,u^\varepsilon(s))dW(s)\rangle+\int_0^t\|\partial_x^2(s,u^\varepsilon(s))\Delta u^\varepsilon(s)\|_{\mathcal{L}_2^Q}^2ds.
\end{eqnarray}
Furthermore,
\begin{eqnarray}\label{theo-2equ6}
&&\sup_{0\leq t\leq T}\|u^\varepsilon(t)\|_2^2+2(1-2\delta)\int_0^T\|u^\varepsilon(s)\|_3^2ds\nonumber\\&&\quad\leq \|u_0\|_2^2 +2C(\delta)\int_
0^T\|u^\varepsilon(s)\|_1^6\|u^\varepsilon(s)\|^4ds\nonumber+2C(\delta)\int_0^T\sup_{0\leq \tau \leq s}\|u^\varepsilon(\tau)\|_2^2ds\\&&\quad\quad+2\sup_{0\leq t\leq T}\int_0^t\langle \Delta^2 u^\varepsilon(s) ,\sigma^\varepsilon(s,u^\varepsilon(s))dW(s)\rangle\nonumber\\&&\quad\quad+\int_0^T\|\partial_x^2(s,u^\varepsilon(s))\Delta u^\varepsilon(s)\|_{\mathcal{L}_2^Q}^2ds.
\end{eqnarray}
By Burkholder-Davis-Gundy inequality, Young inequality, H\"{o}lder inequality and $(\mathbf{H_{3}})$,
\begin{eqnarray}\label{theo-2equ71}
&&\mathbb{E}\sup_{0\leq t\leq T}\int_0^t \langle \Delta^2 u^\varepsilon(s), \sigma^\varepsilon(s,u^\varepsilon(s))dW(s)\rangle\nonumber\\&&
\quad\quad\leq\mathbb{E}\left(\int_0^T\langle \Delta^2 u^\varepsilon(s),\sigma^\varepsilon(s,u^\varepsilon(s))\rangle^2 ds\right)^{\frac{1}{2}}\nonumber\\&&\quad\quad\leq\mathbb{E}\left(\int_0^T\|u^\varepsilon(s)\|_2^2\| \partial_x^2\sigma^\varepsilon(s,u^\varepsilon(s)) \Delta u^\varepsilon(s)\|_{\mathcal{L}_2^Q}^2ds\right)^{\frac{1}{2}}\nonumber\\&&\quad\quad\leq
\mathbb{E}\left(\sup_{0\leq s\leq T}\|\partial_x^2\sigma^\varepsilon(s,u^\varepsilon(s))\Delta u^\varepsilon(s) \|_{\mathcal{L}_2^Q}^2\int_0^T\|u^\varepsilon(s)\|_2^2ds\right)^{\frac{1}{2}}\nonumber\\&&\quad\quad\leq
\frac{C}{2}+\frac{1}{2}\mathbb{E}\int_0^T \|u^\varepsilon(s)\|_2^2ds.
\end{eqnarray}
According to Proposition \ref{tight-pro}, we have
\begin{eqnarray}\label{theo-2equ8}
\int_
0^T\|u^\varepsilon(s)\|_1^6\|u^\varepsilon(s)\|^4ds\leq C_T.
\end{eqnarray}
Plugging (\ref{theo-2equ71}),\;(\ref{theo-2equ8}) into (\ref{theo-2equ6}) yield
\begin{eqnarray*}\label{theo-2equ7}
&&\mathbb{E}\sup_{0\leq t\leq T}\|u^\varepsilon(t)\|_2^2+2(1-2\delta)\mathbb{E}\int_0^T\|u^\varepsilon(s)\|_3^2ds\nonumber\\&&\quad\leq C_T(\mathbb{E}\|u_0\|_2^2+1) +C(\delta)\mathbb{E}\int_0^T\sup_{0\leq \tau \leq s}\|u^\varepsilon(\tau)\|_2^2ds
\end{eqnarray*}
The Gronwall's inequality yields (\ref{theo-2equ}).
\end{proof}
\begin{proposition}\label{tight-2theo1}
 Assume that $(\mathbf{H_{1}})$--$(\mathbf{H_{3}})$,  and $\mathbb{E}\|u_0\|_{V^2}^{2p}\leq C$ for some constant $C >0$ and any $p> 0$ hold. Then, for $0< \delta << \frac{1}{2}$
\begin{eqnarray}\label{theo1-2eequ}
\mathbb{E}\sup_{0\leq t\leq T}\|u^\varepsilon(t)\|_2^{2p}\leq C_{T}\left(\mathbb{E}\|u_0\|_2^{2p}+C\right)+1.
\end{eqnarray}
\end{proposition}
\begin{proof}
For the case $p>1$, we can prove that
\begin{eqnarray}\label{theo1-2equ1}
\mathbb{E}\sup_{0\leq t\leq T}\|u^\varepsilon(t)\|_2^{2p}\leq C_{T}\left(\mathbb{E}\|u_0\|_2^{2p}+C\right).
\end{eqnarray}
From (\ref{theo-2equ5}), we obtain
\begin{eqnarray}\label{theo1-2equ2}
&&\mathbb{E}\sup_{0\leq t\leq T}\|u^\varepsilon(t)\|_2^{2p}\leq C_T(\mathbb{E}\|u_0\|_2^{2p}+\mathbb{E}\sup_{0\leq t \leq T}\left(\int_0^t \|u^\varepsilon(s)\|_2^2ds\right)^p\nonumber\\&&\quad\quad\quad\quad\quad\quad\quad\quad\quad+C(\delta)\mathbb{E}\sup_{0\leq t \leq T}\left(\int_0^t \|u^\varepsilon(s)\|_1^6\|u^\varepsilon(s)\|^4ds\right)^p\nonumber\\&&\quad\quad\quad\quad\quad\quad\quad\quad\quad+
\mathbb{E}\sup_{0\leq t\leq T}\left|\int_0^t\langle \Delta^2 u^\varepsilon(s) ,\sigma^\varepsilon(s,u^\varepsilon(s))dW(s)\rangle\right|^p+1).
\end{eqnarray}
The Burkholder--Davis--Gundy inequality, Young inequality, H\"{o}lder inequality and $(\mathbf{H_{3}})$ yield
\begin{eqnarray}\label{theo1-2equ3}
&&\mathbb{E}\sup_{0\leq t\leq T}\left|\int_0^t \langle \Delta^2 u^\varepsilon(s), \sigma^\varepsilon(s,u^\varepsilon(s))dW(s)\rangle\right|^p\nonumber\\&&
\quad\quad\leq\mathbb{E}\left(\int_0^T\langle \Delta^2 u^\varepsilon(s),\sigma^\varepsilon(s,u^\varepsilon(s))\rangle^2 ds\right)^{\frac{p}{2}}\nonumber\\&&\quad\quad\leq\mathbb{E}\left(\int_0^T\|u^\varepsilon(s)\|_2^2\| \partial_x^2\sigma^\varepsilon(s,u^\varepsilon(s)) \Delta u^\varepsilon(s)\|_{\mathcal{L}_2^Q}^2ds\right)^{\frac{p}{2}}\nonumber\\&&\quad\quad\leq
\mathbb{E}\left(\sup_{0\leq s\leq T}\|\partial_x^2\sigma^\varepsilon(s,u^\varepsilon(s))\Delta u^\varepsilon(s) \|_{\mathcal{L}_2^Q}^2\int_0^T\|u^\varepsilon(s)\|_2^2ds\right)^{\frac{p}{2}}\nonumber\\&&\quad\quad\leq
C\mathbb{E}\left(\int_0^T \|u^\varepsilon(s)\|_2^2ds\right)^{\frac{p}{2}}.
\end{eqnarray}
The Young inequality and H\"{o}lder inequality  yield
\begin{eqnarray*}\label{theo1-2equ4}
&&\left(\int_0^T \|u^\varepsilon(s)\|_2^2 ds\right)^{\frac{p}{2}}\\&&\quad\quad\leq \frac{1}{2}+\frac{1}{2}\left(\int_0^T \|u^\varepsilon(s)\|_2^2 ds\right)^{p}\\&&\quad\quad\leq\frac{1}{2}+\frac{C_T}{2}\int_0^T \|u^\varepsilon(s)\|_2^{2p} ds.
\end{eqnarray*}
The (\ref{theo1-2equ3}) is written as
\begin{eqnarray}\label{theo1-2equ5}
\mathbb{E}\sup_{0\leq t\leq T}\left|\int_0^t \langle \Delta^2 u^\varepsilon(s), \sigma^\varepsilon(s,u^\varepsilon(s))dW(s)\rangle\right|^p
\leq\frac{C}{2}+\frac{C_T}{2}\mathbb{E}\int_0^T \|u^\varepsilon(s)\|_2^{2p} ds.
\end{eqnarray}
The Proposition \ref{tight-pro} yields
\begin{eqnarray}\label{theo1-2equ6}
\mathbb{E}\sup_{0\leq t\leq T}\left(\int_
0^t\|u^\varepsilon(s)\|_1^6\|u^\varepsilon(s)\|^4ds\right)^p\leq C_T.
\end{eqnarray}
The  H\"{o}lder inequality yields
\begin{eqnarray}\label{theo1-2equ7}
\mathbb{E}\sup_{0\leq t\leq T}\left(\int_0^t\|u^\varepsilon(s)\|_2^2ds\right)^p \leq C_T \mathbb{E}\int_0^T \|u^\varepsilon(s)\|_2^{2p}ds.
\end{eqnarray}
Combining (\ref{theo1-2equ5}), (\ref{theo1-2equ6}), (\ref{theo1-2equ7}) with (\ref{theo1-2equ2}) yields
\begin{eqnarray*}\label{theo1-2equ8}
&&\mathbb{E}\sup_{0\leq t\leq T}\|u^\varepsilon(t)\|_2^{2p}\leq C_T\left(\mathbb{E}\|u_0\|_2^{2p}+C\right)\\&&\quad\quad\quad\quad\quad\quad\quad\quad\quad+C_T \mathbb{E}\int_0^T \sup_{0\leq \tau \leq s}\|u^\varepsilon(\tau)\|_2^{2p}ds.
\end{eqnarray*}
The  Gronwall's inequality yields (\ref{theo1-2equ1}).

For the case $0< p \leq 1$,
\begin{eqnarray*}\label{theo1-2equ9}
\|u^\varepsilon(s)\|_2^{2p}\leq \frac{1}{p'}\|u^\varepsilon(s)\|_2^{2pp'}+ \frac{1}{q'}.
\end{eqnarray*}
Here we have used Young inequality
\begin{eqnarray*}\label{theo1-2equ10}
ab\leq \frac{1}{p'} a^{p'}+\frac{1}{q'} b^{q'},
\end{eqnarray*}
where $p',q'>1$\;,$\frac{1}{p'}+\frac{1}{q'}=1$.
Selecting $p'=2p^{-1}$ yields
\begin{eqnarray*}\label{theo1-2equ11}
\|u^\varepsilon(s)\|_2^{2p}\leq \frac{p}{2}\|u^\varepsilon(s)\|_2^4+1.
\end{eqnarray*}
The (\ref{theo1-2equ1}) deduces
\begin{eqnarray*}\label{theo1-2equ12}
\mathbb{E}\sup_{0\leq s\leq T}\|u^\varepsilon(s)\|_2^4\leq C_T\left(\mathbb{E}\|u_0\|_2^{2p}+C\right).
\end{eqnarray*}
The above estimations yields (\ref{theo1-2eequ}).
\end{proof}
To derive the tightness of the solution $\{u^\varepsilon(\cdot,\cdot)\}$, we also need show the ${\rm H\ddot{o}lder}$ continuity of the solution $\{u^\varepsilon(\cdot,\cdot)\}$. Integrating from $s$ to $t$ yields
\begin{eqnarray*}\label{theo1-2equ13}
&&u^\varepsilon(t)-u^\varepsilon(s)=\\&&\quad\quad-\int_s^t(Au^\varepsilon(\tau)+B(u^\varepsilon(\tau)))d\tau+\int_s^tq^\varepsilon(x,\omega)u^\varepsilon(\tau)
d\tau\\&&\quad\quad\quad+\int_s^t\sigma^\varepsilon(\tau,u^\varepsilon(\tau))dW(\tau).
\end{eqnarray*}
\begin{proposition}\label{tight-2theo2}
 Assume that $(\mathbf{H_{1}})$--$(\mathbf{H_{3}})$,  and $u_0\in {V^2}$ hold.  For any $T>0$, there is  a constant $C_T>0$ such that
\begin{eqnarray}\label{theo2-equ}
\mathbb{E}\| u^\varepsilon(t)-u^\varepsilon(s)\|_1^2\leq  C_T \mid t-s\mid^{\frac{1}{2}},
\end{eqnarray}
where $s,t\in[0,T]$.
\end{proposition}
\begin{proof}
By  estimation (\ref{theo-equ}), estimation (\ref{theo1-2eequ}), $(\mathbf{H_2})$ and $(\mathbf{H_3})$, we obtain
\begin{eqnarray*}\label{theo2-equ1}
\mathbb{E}\left\|\int_s^t -A u^\varepsilon(\tau)d\tau\right\|^2\leq C|t-s|^2\mathbb{E}\sup_{0\leq t\leq T}\|u^\varepsilon(t)\|_2^2\leq C|t-s|^2,
\end{eqnarray*}
\begin{eqnarray*}\label{theo2-equ2}
\mathbb{E}\left\|\int_s^t-B(u^\varepsilon(\tau))d\tau\right\|^2\leq C|t-s|^2\mathbb{E}\sup_{0\leq t \leq T}\|u^\varepsilon(t)\|_2^4\leq C|t-s|^2,
\end{eqnarray*}
where we have used $B(u)\leq C \|u\|_2\|u\|_1$, $C$ is a constant.
\begin{eqnarray*}\label{theo2-equ3}
\mathbb{E}\left\|\int_s^t q^\varepsilon(x,\omega)u^\varepsilon(\tau)d\tau\right\|^2\leq C|t-s|^2\sup_{0\leq t \leq T}\|u^\varepsilon(t)\|^2\leq C|t-s|^2,
\end{eqnarray*}
and
\begin{eqnarray*}\label{theo2-equ4}
\mathbb{E}\left\|\int_s^t \sigma^\varepsilon(\tau,u^\varepsilon(\tau)dW(\tau)\right\|^2\leq C|t-s|.
\end{eqnarray*}
Thus,
\begin{eqnarray}\label{theo2-equ5}
\mathbb{E}\| u^\varepsilon(t)-u^\varepsilon(s)\|^2&\leq &C|t-s|+C|t-s|^2\\&\leq& C|t-s|(1+T).
\end{eqnarray}
It follows from Lemma \ref{lemma-ine1} that
\begin{eqnarray*}\label{theo2-equ6}
\|u\|_1\leq C\|u\|_2^{\frac{1}{2}}\|u\|^{\frac{1}{2}}.
\end{eqnarray*}
Thereby,
\begin{eqnarray*}\label{theo2-equ7}
&&\mathbb{E}\| u^\varepsilon(t)-u^\varepsilon(s)\|_1^2\\&&
\quad\quad\leq C \mathbb{E}\|u^\varepsilon(t)-u^\varepsilon(s)\|_2 \|u^\varepsilon(t)-u^\varepsilon(s)\|
\\&&\quad\quad\leq C (\mathbb{E}\|u^\varepsilon(t)-u^\varepsilon(s)\|_2^2)^{\frac{1}{2}}(\mathbb{E}\|u^\varepsilon(t)-u^\varepsilon(s)\|^2)^{\frac{1}{2}}
\\&&\quad\quad\leq C (\mathbb{E}\sup_{0\leq t\leq T}\|u^\varepsilon(t))\|_2^2)^{\frac{1}{2}}(\mathbb{E}\|u^\varepsilon(t)-u^\varepsilon(s)\|^2)^{\frac{1}{2}}
\\&&\quad\quad\leq C \mid t-s\mid^{\frac{1}{2}}(1+T)^{\frac{1}{2}}.
\end{eqnarray*}
\end{proof}
By Lemma~{\ref{lem4}}, we have the following result.
\begin{lemma}\label{lem1}
 Assume $(\mathbf{H_1})$--$(\mathbf{H_3})$, $\mathbb{E}\|u_0\|_{V^2}^2\leq C$  for some constant $C >0$ hold. For every $T>0$\,, the family $\{u^{\varepsilon}(\cdot,\cdot)\}_{0<\epsilon\leq1}$ is tight in space $C(0,T; X^1)$.
\end{lemma}
\section{Proof of Theorem \ref{main-theom1}}\label{sec:lim}
In this section, we give the proof  of Theorem \ref{main-theom1}.  By the tightness of $\{u^{\varepsilon}(\cdot,\cdot)\}$, there is a subsequence converges in distribution to $C(0,T; X^1)$, the subsequence is written as $\{{u}^{\varepsilon}(\cdot,\cdot)\}$. By the Skorohod theorem \cite{Bi} one can construct a new probability space and new variable without changing the distribution such that $\{u^{\varepsilon}(\cdot,\cdot)\}$ (here we don't change the notations) converges almost surely to $u(\cdot,\cdot)$ in space  $C(0,T; X^1)$.  Next, we determine the limit process $u(\cdot,\cdot)$.

Before presenting the proof of Theorem \ref{main-theom1}, we first show the following  lemmas.
\begin{lemma}\label{pf-lem}
For all  $\varphi(x)\in C_0^\infty(\mathbb{T}^2)$, $\phi(t)\in C^\infty (0,T)$ such that
\begin{eqnarray*}
\lim_{\varepsilon\to 0}\int_0^T\int_{\mathbb{T}^2}(q(\mathbf{T}_{\frac{x}{\varepsilon}}\omega) u^{\varepsilon}(x,t))\varphi(x)\phi(t)dxdt=\int_0^T\int_{\mathbb{T}^2}\bar{q}u(x,t)\varphi(x)\phi(t)dxdt.
\end{eqnarray*}
\end{lemma}
\begin{proof}
\begin{eqnarray}\label{pf-equ1}
&&\int_0^T\int_{\mathbb{T}^2} q(\mathbf{T}_{\frac{x}{\varepsilon}}\omega)u^{\varepsilon}(x,t)\varphi(x)\phi(t)dxdt-\int_0^T\int_{\mathbb{T}^2} \bar{q}u(x,t)\varphi(x)\phi(t)dxdt\nonumber\\&&\quad\quad\quad\quad=
\int_0^T\int_{\mathbb{T}^2} q(\mathbf{T}_{\frac{x}{\varepsilon}}\omega)(u^{\varepsilon}(x,t)-u(x,t))\varphi(x)\phi(t)dxdt\nonumber\\&&\quad\quad\quad\quad\quad\quad
+\int_0^T\int_{\mathbb{T}^2}(q(\mathbf{T}_{\frac{x}{\varepsilon}}\omega)-\bar{q})u(x,t)\varphi(x)\phi(t)dxdt.
\end{eqnarray}
By assumption $(\mathbf{H_2})$, (\ref{theo-equ}),~Lemma \ref{lem12} and the Birkhoff ergodic theorem~\cite[Theorem 7.2]{Zhi1}, (\ref{pf-equ1}) vanishes as $\varepsilon\to 0$. The proof of Lemma \ref{pf-lem} is complete.
\end{proof}
\begin{lemma}\label{pf-lem1}
Assume $(\mathbf{H_{3}})$ holds. For every $T >0$, almost sure $\omega\in\Omega$,
$\varphi(x)\in C_0^\infty(\mathbb{T}^2)$, $\phi(t)\in C^\infty (0,T)$ such that
\begin{eqnarray}\label{pf-lemm1-equ}
\lim_{\varepsilon\rightarrow0}\int_0^T\int_{\mathbb{T}^2} (\sigma^\varepsilon(t,u^\varepsilon(x,t))\varphi(x)\phi(t)dxdW_1(t)=\int_0^T\int_{\mathbb{T}^2}\bar{\sigma}(u(x,t)))\varphi(x)\phi(t) dxdW_1(t).
\end{eqnarray}
\end{lemma}
\begin{proof}
The (\ref{pf-lemm1-equ}) is written as
\begin{eqnarray*}\label{pf-lemm1-equ2}
&&\int_0^T\int_{\mathbb{T}^2} (\sigma^\varepsilon(t,u^\varepsilon(x,t))-\bar{\sigma}(u^\varepsilon(x,t))+\bar{\sigma}(u^\varepsilon(x,t))-\bar{\sigma}(u(x,t)))\varphi(x)\phi(t) ) dx dW(t)
\\&&=\int_0^T\int_{\mathbb{T}^2} (\sigma^\varepsilon(t,u^\varepsilon(x,t))-\bar{\sigma}(u^\varepsilon(x,t)))\varphi(x)\phi(t)  dx dW(t)\\&&\quad\quad+\int_0^T\int_{\mathbb{T}^2} (\bar{\sigma}(u^\varepsilon(x,t))-\bar{\sigma}(u(x,t)))\varphi(x)\phi(t)  dx dW(t)\\&&\triangleq I_1+I_2
\end{eqnarray*}
For $I_1$, we divide the time interval $[0,T]$ into small interval of size $\delta > 0$, i.e. $0=t_0< t_1 < \cdots <t_{[\frac{T}{\delta}]}+1=T$, $t_k=k\delta, k=0,1,\cdots, [\frac{T}{\delta}]$, for all $t \in [t_k,t_{k+1}]$, we obtain
\begin{eqnarray*}\label{pf-lemm1-equ3}
&&\int_{t_k}^t\int_{\mathbb{T}^2}  (\sigma^\varepsilon(s,u^\varepsilon(x,s))-\sigma^\varepsilon(s,u^\varepsilon(x,t_k))+\sigma^\varepsilon(s,u^\varepsilon(x,t_k))-\bar{\sigma}(u^\varepsilon(x,t_k))\\&& \quad\quad\quad+\bar{\sigma}(u^\varepsilon(x,t_k))-\bar{\sigma}(u^\varepsilon(x,s))+\bar{\sigma}(u^\varepsilon(x,s))-\bar{\sigma}(u(x,s)))\varphi(x)\phi(s)dxdW(s)
\\&&=\int_{t_k}^t\int_{\mathbb{T}^2}  (\sigma^\varepsilon(s,u^\varepsilon(x,s))-\sigma^\varepsilon(s,u^\varepsilon(x,t_k)))\varphi(x)\phi(s)dxdW(s)\\&&\quad\quad+\int_{t_k}^t\int_{\mathbb{T}^2} (\sigma^\varepsilon(s,u^\varepsilon(x,t_k))-\bar{\sigma}(u^\varepsilon(x,t_k)))\varphi(x)\phi(s)dxdW(s)
\\&&\quad\quad+\int_{t_k}^t\int_{\mathbb{T}^2} (\bar{\sigma}(u^\varepsilon(x,t_k))-\bar{\sigma}(u^\varepsilon(x,s)))\varphi(x)\phi(s)dxdW(s)\\&&\quad\quad+\int_{t_k}^t\int_{\mathbb{T}^2}  (\bar{\sigma}(u^\varepsilon(x,s))-\bar{\sigma}(u(x,s)))\varphi(x)\phi(s)dxdW(s)
\\&& \triangleq I_{11}+I_{12}+I_{13}+I_{14}.
\end{eqnarray*}
For $I_{11}$, thanks to the fact that  $\delta > 0$ is enough small, by (\ref{theo2-equ5}) and $(\mathbf{H_3})$, we have  $I_{11}$ vanishes as $\varepsilon\to 0$. For $I_{12}$, due to the fact that $\sigma(t,\cdot)$ is periodic with respect to  $t$,  the~\cite[Theorem 2.6]{Ci} yields $I_{12}$  converges to $0$ as $\varepsilon\to 0$. For $I_{13}$, thanks to the fact that  $\delta > 0$ is enough small,  the (\ref{theo2-equ5}) and $(\mathbf{H_3})$ yield $I_{13}$ vanishes  as $\varepsilon\to 0$. For $I_{14}$, $I_{14}$ converges to $0$ as $\varepsilon\to 0$ by  $(\mathbf{H_3})$. Thus $I_1$ converges to $0$ as $\varepsilon \to 0$. By  $(\mathbf{H_3})$, we have $I_{2}$  vanishes as $\varepsilon\to 0$.
The proof of Lemma \ref{pf-lem1} is complete.
\end{proof}
\begin{lemma}\label{pf-lem2}
For all  $\varphi(x)\in C_0^\infty(\mathbb{T}^2)$, $\phi(t)\in C^\infty (0,T)$ such that
\begin{eqnarray*}\label{pf-lem2-equ}
\lim_{\varepsilon\to 0}\int_0^T\int_{\mathbb{T}^2}B(u^\varepsilon(x,t))\varphi(x)\phi(t)dxdt=\int_0^T\int_{\mathbb{T}^2}B(u(x,t))\varphi(x)\phi(t)dxdt.
\end{eqnarray*}
\end{lemma}
\begin{proof}
\begin{eqnarray}\label{pf-lem2-equ1}
&&\int_0^T\int_{\mathbb{T}^2}(B(u^\varepsilon(x,t))-B(u(x,t)))\varphi(x)\phi(t)dxdt\nonumber\\&&\quad\quad\leq \int_0^T\|B(u^\varepsilon(x,t))-B(u(x,t))\|_{-1}\|\varphi(x)\|_1\phi(t)dt.
\end{eqnarray}
By $(v)$ of Lemma \ref{lemma-b} and (\ref{theo-eequ}), we deduce (\ref{pf-lem2-equ1}) converges to $0$ as $\varepsilon\to 0$. The proof of Lemma \ref{pf-lem2} is complete.
\end{proof}

\begin{proof}[\textbf{Proof of Theorem \ref{main-theom1}}]
Multiplying both sides of the equations (\ref{p-main-equ1}) by the test function $\varphi(x,t)\in C^1(0,T; H^1(\mathbb{T}^2))$ with $\varphi(x,T)=0,\;\varphi(x,0)=1$ yields
\begin{eqnarray}\label{main-theo-equ}
&&-\int_0^T\int_{\mathbb{T}^2}u^\varepsilon(x,s)\frac{\partial\varphi(x,s)}{\partial s}dxds\nonumber\\&&\quad=\int_{\mathbb{T}^2} u_0(x)dx\nonumber\\&&\quad\quad-\int_0^T\int_{\mathbb{T}^2}Au^\varepsilon(x,s)\varphi(x,s)dxds\nonumber\\&&\quad\quad
-\int_0^T\int_{\mathbb{T}^2}B(u^\varepsilon(x,s))\varphi(x,s)dxds\nonumber\\&&\quad\quad
+\int_0^T\int_{\mathbb{T}^2}q^\varepsilon(\textbf{T}_{\frac{x}{\varepsilon}}\omega)u^\varepsilon(x,s)\varphi(x,s)dxds\nonumber\\&&\quad\quad
+\int_0^T\int_{\mathbb{T}^2}\sigma^\varepsilon(s,u^\varepsilon(x,s))\varphi(x,s)dxdW(s).
\end{eqnarray}
Passing the limit $(\varepsilon\to 0) $ on the both sides of (\ref{main-theo-equ}) yields
\begin{eqnarray}\label{main-theo-equ1}
&&-\int_0^T\int_{\mathbb{T}^2}u(x,s)\frac{\partial\varphi(x,s)}{\partial s}dxds\nonumber\\&&\quad=\int_{\mathbb{T}^2} u_0(x)dx\nonumber\\&&\quad\quad-\int_0^T\int_{\mathbb{T}^2}Au(x,s)\varphi(x,s)dxds\nonumber\\&&\quad\quad
-\int_0^T\int_{\mathbb{T}^2}B(u(x,s))\varphi(x,s)dxds\nonumber\\&&\quad\quad
+\int_0^T\int_{\mathbb{T}^2}\bar{q}u(x,s)\varphi(x,s)dxds\nonumber\\&&\quad\quad
+\int_0^T\int_{\mathbb{T}^2}\bar{\sigma}(u(x,s))\varphi(x,s)dxdW(s).
\end{eqnarray}
Here, we have used Lemma \ref{pf-lem}, Lemma \ref{pf-lem1} and Lemma \ref{pf-lem2}. That is,
\begin{eqnarray*}\label{main-theo-equ2}
\begin{cases}
{u}_{t}(t)+[A u(t)+B(u(t))]=\bar{q}u(t)+\bar{\sigma}(u(t))dW(t),\\
\mathrm{div}\; u(t)=0,\\
u(0)=u_0(x).
\end{cases}
\end{eqnarray*}
We finish the proof of Theorem  \ref{main-theom1}.
\end{proof}

\textbf{Acknowledgements}\\
 The authors would like to thank all anonymous commenters for  their views and helpful suggestions for improvement. This work is supported  by the National Natural Science Foundation of China (Nos. 12401304, 12271293), University Natural Science Research Project of Anhui Province(No. 2022AH050404), the project of Youth Innovation Team of Universities of Shandong Province (No. 2023KJ204) and  Shandong Provincial Natural Science Foundation (Nos. ZR2023MA002, ZR2024MA069).

\textbf{Author Declarations}

\textbf{Conflict of interest}\\   The authors have no conflicts to disclose.

\textbf{Data availability}\\
Data sharing not applicable to this article as no data sets were generated or analyzed during the current study.

\end{document}